\documentclass[regno,11pt]{amsart}

\usepackage[english]{babel}
\usepackage{amssymb,verbatim,microtype,graphicx,hyperref}
\usepackage[T1,T5]{fontenc}
\usepackage{mathtools}
\mathtoolsset{showonlyrefs}

 \usepackage{xcolor}

\newtheorem{theorem}{Theorem}[section]
\newtheorem{proposition}[theorem]{Proposition}

\newtheorem{corollary}[theorem]{Corollary}

\newtheorem*{thmA}{Theorem A}
\newtheorem*{thmB}{Theorem B}
\newtheorem*{thmC}{Theorem C}
\newtheorem*{thmD}{Theorem D}

\newtheorem*{corollary*}{Corollary}
\theoremstyle{definition}

\newtheorem{example}[theorem]{Example}
\newtheorem{remark}[theorem]{Remark}

\newtheorem*{ackn}{Acknowledgements}

\numberwithin{equation}{section}
 \newcommand{\beq}{\begin{equation}}
\newcommand{\eeq}{\end{equation}}

\newcommand{\Cn}{{\mathbb C}^n}
\newcommand{\C}{\mbox{$\mathbb{C}$}}

\newcommand{\R}{\mbox{$\mathbb{R}$}}
\newcommand{\D}{\mbox{$\mathbb{D}$}}
\newcommand{\B}{\mbox{$\mathbb{B}$}}
\newcommand{\E}{\mbox{$\mathcal{E}$}}
\newcommand{\J}{\mbox{$\mathcal{J}$}}
\newcommand{\cL}{\mbox{$\mathcal{L}$}}

\newcommand{\JP}{\mbox{$\mathcal{JP}$}}

\newcommand{\supp}{\mbox{supp}}
\newcommand{\SH}[1]{\mbox{$\mathcal{SH}(#1)$}}
\newcommand{\SHc}[1]{\mbox{$\mathcal{SH}^c(#1)$}}
\newcommand{\PSH}[1]{\mbox{$\mathcal{PSH}(#1)$}}

\newcommand{\PSHc}[1]{\mbox{$\mathcal{PSH}^c(#1)$}}

\newcommand{\cA}{\mathcal{A}}
\newcommand{\wz}{\widehat z}
\newcommand{\wu}{\widehat u}
\newcommand{\wv}{\widehat v}

\begin{document}

\title{Regularity of rooftop envelopes}

\author{Eleonora Di Nezza}\address{Dipartimento di Matematica, Università degli Studi di Roma Tor Vergata, Via della Ricerca Scientifica 1, 00133 Roma, Italy}\email{dinezza@mat.uniroma2.it}

\author
{Alexander Rashkovskii}\address{Department of Mathematics and Physics, University of Stavanger, 4036 Stavanger, Norway}\email{alexander.rashkovskii@uis.no}

\date{}

\keywords{Rooftop envelope, plurisubharmonic geodesic, Monge-Amp\`ere equation, Jensen measure}
\subjclass[2020]{Primary 32U05, 32W20; Secondary 32F17, 32T15, 32U10, 53C22}

\begin{abstract}
We study continuity, H\"older regularity, and $C^{1,1}$-regularity of geodesics between continuous plurisubharmonic functions on bounded domains of $\Cn$. We then derive regularity properties of rooftop envelopes.
\end{abstract}

\maketitle

\section{Introduction}

The subject of geodesics in the space of K\"{a}hler potentials on a compact K\"ahler manifold $(X, \omega)$ originates from the pioneering work by Mabuchi \cite{Mab} and had since become a central topic in K\"{a}hler geometry. It is well known that such geodesics can be characterized as solutions to a homogeneous complex Monge-Amp\`ere equation \cite{Sem}.
In recent years, substantial progress has been achieved through the contributions of Guedj, Zeriahi, Berndtsson, Darvas, Di Nezza, Lu, and others. 

A particularly effective approach to connecting quasi-plurisubharmonic functions by geodesics is based on the use of  \emph{rooftop envelopes}. In this framework the geodesic joining two bounded $\omega$-plurisubharmonic functions $\varphi_0, \varphi_1$ is $\varphi_t$, $t\in (0,1)$ which can be described as an envelope on $M=X\times \cA$, $\cA=\{0<\log|\zeta|<1\}\subset\C$,
$$\varphi_{\log|z|} (x):= \Phi(x,z)=\sup \{U\in \PSH{M, \tilde\omega}:\; U|_{\partial M} \leq \varphi_0, \varphi_1\},$$
where $\tilde\omega$ is the pullback of $\omega$ to $M$.
This viewpoint has proved to be very fruitful and has been further developed in several recent works (\cite{RWN}, \cite{Da14a}, \cite{DNL18}--\cite{DNL21}). An overview of the subject can be found in \cite{DNL23}.
\smallskip

Similar questions in the local setting, namely for plurisubharmonic functions defined on bounded domains in $\Cn$, have been studied in \cite{A}, \cite{AD}, \cite{BB},\cite{R16}--\cite{R22} and other works, and this remains an active area of research. We refer to \cite{R23} for a recent survey. 

In this context, the geodesic between two plurisubharmonic ({\it psh}) functions $u_0$ and $u_1$ arises as the upper envelope of plurisubharmonic functions on a product domain and can be described as the supremum of all subgeodesics. More precisely, denote $\cA_j=\{\log|\zeta|=j\}$, $j=0,1$ and consider the class
$W(u_0,u_1)$ of functions $v\in\PSH{\Omega\times\cA}$ such that $ \limsup v(\cdot,\zeta)\le u_j(\cdot)$ as $\zeta\to \cA_j$, $j=0,1$. 
The upper envelope is defined as
\begin{equation*}
\widehat u(z,\zeta)=\sup\{v(z,\zeta):\: v\in W(u_0,u_1)\}
\in \mathcal{PSH}(\Omega\times\cA).
\end{equation*}

We say that a family $v_t(z)$, $0<t<1$, is a {\it subgeodesic for $u_0,u_1$} if $v_{\log|\zeta|}(z)\in W(u_0,u_1)$, and the largest subgeodesic, $u_t$, is by definition the {\it geodesic} joining $u_0$ and $u_1$.
In other words, 
\[u_t(z)=u_{\log|\zeta|}(z):=\widehat u(z,\zeta).\] 

\smallskip

In this paper, we focus on the regularity properties of geodesics in the local setting. More generally, we study \emph{psh envelopes} of the form
$$P(h)=\sup \{v\in\PSH\Omega:\: v\le h \}^*,$$ 
where $h$ is a measurable function locally bounded from above and $^*$ denotes the u.s.c. regularization, and {\it rooftop envelopes} $P(u,v)=P(\min\{u,v\})$.

It is worth emphasizing that even some basic questions concerning the regularity of these envelopes remain open. The main difficulties  stem from boundary conditions and lack of control of the total Monge-Amp\`ere mass.


\smallskip

Continuity of $P(h)$ for $h\in C(\overline\Omega)$ in B-regular domains is a classical result due to Walsh \cite{W}. For a larger class of domains, this can be done by using Jensen measures and Edwards' theorem; we refer among other to 
\cite{Si}, \cite{W01}.



\smallskip

We extend Edwards' theorem to the setting of $\omega$-psh functions on compact K\"ahler manifolds $(X,\omega)$, by arriving to the notion of {\it Jensen pairs}, see Theorem~\ref{thm:Edwards}. This allows us to provide a new and more elementary proof of the continuity of envelopes $P_\omega(h)$ in this setting (Theorem~\ref{thm:omega_cont}).

Then we concentrate on
the regularity properties of geodesics in the local case. 
In contrast to \cite{A}, which mainly deals with the case of zero boundary values, we consider arbitrary continuous boundary data and provide a complete description of the possible boundary behavior using Jensen measure techniques:

\begin{thmA}
 Let $\Omega$ be a bounded hyperconvex domain and $u_0,u_1\in\PSH\Omega\cap L^\infty(\Omega)$. Then 
 \begin{itemize}
     \item[(i)] $\widehat u = P_{\Omega\times\cA}(\Psi_u)$, where $\Psi_u(z,\zeta)=(1-\log|\zeta|)\,u_0(z)+\log|\zeta|\,u_1(z)$;
     \smallskip
     \smallskip
      \item[(ii)] if $u_0,u_1\in \PSH{\Omega}\cap C(\overline\Omega)$ and $\Omega$ is B-regular, then $\widehat u \in C(\overline{\Omega\times \cA})$ and $\widehat u=\Psi_u$ on $\partial \Omega \times \cA$.
 \end{itemize}
\end{thmA}

The first item is Proposition ~\ref{prop:env_geod} and the second one is Theorem~\ref{thm:geo_bv}.


\smallskip
Furthermore, we analyze how the regularity of geodesics depends on the boundary data. We summarize our findings as follows (Theorem~\ref{thm:1/2 H}):

\begin{thmB}
 Let  $\Omega$ be strongly pseudoconvex. Then the followings hold:
 \begin{itemize}
 \item[(i)] Assume $u_0,u_1\in\PSH\Omega$ are $\alpha$-H\"older continuous on $\overline\Omega$, $0<\alpha\le 1$. Then the geodesic $u_t$ is $\frac\alpha2$-H\"older on $\overline{\Omega}$, uniformly in $t\in [0,1]$.
 
 \item[(ii)] Assume $u_0,u_1\in C^{1,1}(\overline \Omega)$ and $u_0-u_1=const$ on $\partial \Omega$. Then $u_t$ is Lipschtiz on $\overline{\Omega}$, uniformly in $t\in [0,1]$.
 \end{itemize}
\end{thmB} 

 In the special case of $\Omega=\B_n$, the unit ball, we prove that the geodesic is locally $C^{1,1}$-smooth in $\Omega$ and its second order derivatives are locally bounded (Theorem~\ref{thm:C11 geod}): 
\begin{thmC}\label{thm:C11 geod}
   Let $u_0,u_1\in C^{1,1}(\overline\B_n)$. Then the function $\wu$ has second-order partial derivatives in the $z$-variables almost everywhere in $\B_n\times\cA$, and the derivatives are locally bounded. More precisely, for any $K\Subset\B_n$,
   \[
   \left|\frac{\partial^2 \wu(z,\zeta)}{\partial z_j\partial \bar z_k}
   \right|\le \frac{C}{{\rm dist}\, (K,\partial\B_n)^2}, \quad z\in K,\ \zeta\in \cA
   \]
   for a positive constant $C>0$ independent of $K$, $z$ and $\zeta$.
\end{thmC}

As a further contribution, we establish a local version of the ``Darvas-Rubinstein formula", which relates geodesics and envelopes via a Legendre transform (Theorem~\ref{thm:DRrepresentation}).

\begin{thmD}
Let $\Omega$ be a bounded domain. Let $u_0,u_1\in\PSH\Omega\cap L^\infty(\Omega)$ and $u_t$ be the geodesic joining them. Then 
\begin{equation}\label{eq:DRrepresentation_intro}
u_t(z)=\sup_{C\in \R} \{P^C(z)+Ct \}, \ 0<t<1,\ z\in\Omega,
\end{equation}
where $P^C:=P(u_0,u_1-C)$. 
Conversely,
\begin{equation}\label{eq:DRenv_intro}
P^C(z)=\inf_{t\in (0,1)}\{u_t(z)-Ct \}=:u^C(z),\quad C\in\R,\ z\in\Omega.
\end{equation}
In particular, 
\begin{equation*}
P(u_0,u_1)=\inf_{t\in (0,1)} u_t.
\end{equation*}
Moreover, if $\Omega$ is B-regular and $u_0,u_1\in\PSH\Omega\cap C(\overline\Omega)$, then \eqref{eq:DRrepresentation_intro} and \eqref{eq:DRenv_intro} hold on $\partial\Omega$ as well.
\end{thmD}

The above correspondence allows us to transfer regularity properties between the two settings and yields several corollaries concerning the regularity of envelopes:

\begin{corollary*}
 Let $\Omega$ be strongly pseudoconvex. We have
 \begin{itemize}
     \item[(i)] If $u,v\in\PSHc\Omega$ be $\alpha$-H\"older continuous in $\overline\Omega$, $0<\alpha\le 1$, then $P(u,v)$ is $\frac\alpha2$-H\"older continuous in $\overline\Omega$. 
\item[(ii)] If $u,v\in\PSH\Omega\cap C^{1,1}(\overline\Omega)$, $u-v=const$ on $\partial\Omega$, then $P(u,v)$ is Lipschitz in $\overline\Omega$.
\item[(iii)] If $u,v\in C^{1,1}(\overline\B_n)$, their envelope $P(u,v)$ belongs to $C^{1,\bar 1}(\B_n)$.
  \end{itemize}
\end{corollary*}
This is proved in Corollary~\ref{cor:H env} and Corollary~\ref{thm:smooth env}. We mention that in Theorem~\ref{thm:smooth env 2} we establish a similar result as in the last item for envelopes of type $P(\psi)$ for some $\psi\in C^{1,1} (\overline{\B}_n)$. More precisely, we show that $P(\psi) \in C^{1,1} (\B_n)$. The proof of such a result uses analytic disks.

\smallskip

For completeness - and for comparison with the compact case - we recall that on a compact K\"ahler manifold $X$, it was shown in \cite{T18} that if $f\in C^{1,1}(X)$, then the envelope $P_\omega(f)$ also belongs to $C^{1,1}(X)$. The more general setting of a big cohomology class $\{\theta\}$ was later addressed in \cite{DNT}, where it was proved that if $f\in C^{1,\bar{1}}(X)$, then the envelope $P_\theta(f)$ likewise lies in $C^{1,\bar{1}}(X)$. \\
This latter result provides evidence that analogous regularity properties should hold for envelopes on bounded (strongly pseudoconvex) domains. Establishing such a result in this generality remains an open problem.

\begin{ackn}
The authors are grateful to Sławomir Dinew for sharing his ideas concerning the proof of Theorem~\ref{thm:1/2 H}. They also thank Mårten Nilsson, Nihat G\"okhan Göğüş for pointing out an issue with a non-compact version of Edwards' theorem in \cite{GogPerPol}, and Stefano Trapani and Ahmed Zeriahi for their useful comments. The first-named author is supported by the ERC grant SiGMA (No. 101125012, PI: Eleonora Di Nezza). Part of this work was carried out during the visit of the second-named author to the Institut de Mathématiques de Jussieu (Sorbonne Universit\'e) and he acknowledges the institution’s support.
\end{ackn}

\section{Preliminaries}\label{sec: pre}
Let $\Omega$ be a domain in $\Cn$. We denote by $\PSH\Omega$ the set of all plurisubharmonic (psh for short) functions in $\Omega$ and 
$\PSHc\Omega:=\PSH\Omega\cap C(\overline\Omega)$.

A domain $\Omega$ is:
\begin{enumerate}
 \item[(i)]   \emph{pseudoconvex} if there exists $\rho\in\PSH\Omega$ such that $\{z \in \Omega : \rho(z) < c \} \Subset \Omega$ for all $c \in\R$ (i.e., $\rho(z)\to+\infty$ as $z\to\partial\Omega$);
\item[(ii)]  \emph{hyperconvex} if there exists $\rho \in \PSH\Omega$ such that $\{z \in \Omega : \rho(z) < c \} \Subset \Omega$ for all $c < 0$ (i.e., $\rho(z)\to 0$ as $z\to\partial\Omega$); we note that one can always choose $\rho\in\PSH\Omega\cap C^\infty(\Omega)$ \cite{Bl99};
\item[(iii)]  
\emph{$B$-regular} if every point $p\in\partial\Omega$ has a strong psh barrier, i.e., a function $\rho_p\in\PSH\Omega$ such that $\sup_\Omega \rho_p=0$ while $\sup_{\Omega\setminus U} \rho_p<0$ for any neighborhood $U$ of $p$; this function can be chosen to belong to $\PSHc{\Omega}$ \cite{Si}. It is also proved in \cite{Si} that $\Omega$ is $B$-regular if and only if for any $f\in C(\partial\Omega)$ there exists $u\in\PSH\Omega\cap C(\overline\Omega)$ such that $u=f$ on $\partial\Omega$. 
 \item[(iv)] \emph{strongly pseudoconvex} if there exists a strictly psh function $\rho$ in a neighborhood of $\overline\Omega$ such that $\Omega=\{\rho<0\}$ and $\nabla \rho\neq 0$ on $\partial\Omega$.
\end{enumerate}
We have (iv)$\Rightarrow$(iii)$\Rightarrow$(ii)$\Rightarrow$(i). Also, any pseudoconvex domain with Lipschitz boundary is hyperconvex \cite{Dem87}. We also stress that product domains (in particular, polydisks) are not $B$-regular.

\medskip
Let now $\cA=\{0<\log|\zeta|<1\}\subset\C$ be the annulus bounded by the circles $\cA_j=\{\log|\zeta|=j\}$, $j=0,1$. Given $u_0,u_1\in\PSH{\Omega}$, we consider the class
$W(u_0,u_1)$ of functions $v\in\PSH{\Omega\times\cA}$ such that $ \limsup v(\cdot,\zeta)\le u_j(\cdot)$ as $\zeta\to \cA_j$, $j=0,1$. Note that $\Omega\times\cA\subset\C^{n+1}$ is not $B$-regular.

When both $u_0$ and $u_1$ are bounded from above, $W(u_0,u_1)\neq\emptyset$ because it contains $u_0+u_1-M$ for a constant $M$ big enough; the upper envelope 
\beq\label{eq:hat u}
\widehat u(z,\zeta)=\sup\{v(z,\zeta):\: v\in W(u_0,u_1)\}
\in \mathcal{PSH}(\Omega\times\cA)
\eeq
(no upper regularization needed) is independent of $\arg\zeta$; in particular, it is a convex function of $t:=\log|\zeta|$. 
\smallskip
Indeed, let $v\in W(u_0, u_1)$ and write $\zeta=e^{t+i\theta}$, $\theta\in [0, 2\pi)$. We define
$$\tilde{v} (x, |\zeta|):= \left(\sup_{\theta\in [0, 2\pi)} v\left(z, e^{t+i\theta}\right) \right)^*.$$
We have $\tilde{v}\in \PSH {\Omega \times \cA}$, since it is the regularized supremum of locally bounded above psh functions and by construction $\tilde{v} \geq v$. Also, since $\sup_{\theta} v\left(z, e^{t+i\theta} \right)\leq u_j (z) $, we infer that $\tilde{v}(\cdot, |\zeta|) \leq u_j(\cdot)$ as $\zeta\rightarrow  \mathcal{A}_j$. This means that in the envelope in \eqref{eq:hat u} we are allowed to consider only candidates which are independent of $\arg\zeta$.

\medskip
We say that a family $v_t(z)$, $0<t<1$, is a {\it subgeodesic for $u_0,u_1$} if $v_{\log|\zeta|}(z)\in W(u_0,u_1)$, and the largest subgeodesic, $u_t$, is by definition the {\it geodesic} joining $u_0$ and $u_1$.
In other words, 
\[u_t(z)=u_{\log|\zeta|}(z):=\widehat u(z,\zeta).\] 
Note that, since $\hat{u} \in \mathcal{PSH}(\Omega\times \mathcal{A})$, it follows that $t\rightarrow u_t$ is convex. Moreover, $\widehat u$ is a maximal psh function since it is a Perron envelope,  which in the case of bounded functions $u_0$, $u_1$ means that it satisfies the homogeneous Monge-Ampère equation 
\begin{equation}\label{eq: MA_geo}
(dd^c\widehat u)^{n+1}=0 \quad \qquad  \textit{in} \; \,\Omega\times\cA,
\end{equation}

Furthermore, if $u_0,u_1\in\PSH\Omega\cap L^\infty(\Omega)$, one can construct a subgeodesic 
\[
V_t:=\max\{u_0-M\,t,u_1-M\,(1-t)\},
\]
where $M=\|u_0-u_1\|_\infty$ \cite{Bern}, and get the bounds
\beq\label{VtM} V_t\le u_t\le (1-t)u_0+tu_1, \quad t\in (0,1),\eeq
where the second inequality is a consequence of the convexity of $u_t$ in $t$, and it holds for any $u_0,u_1$. In particular, we have 
 $u_t\to u_j$, uniformly on $\Omega$, as $t\to 0,1$.
 
 Combining this information with \eqref{eq: MA_geo} we can infer that $\widehat{u}$ is a solution of 
\begin{equation}\label{MA_geo}
\begin{cases}
(dd^c\widehat u)^{n+1}=0 \quad \qquad  \textit{in} \; \,\Omega\times\cA,\\
 \limsup_{\zeta \to \cA_j} \widehat u = u_j, \quad j=0,1.
\end{cases}
\end{equation}

Another, and very important, type of subgeodesics (that works also in the unbounded case) can be constructed in terms of the rooftop envelopes \cite{Da14a}:
\beq\label{Darvas_sg} 
v_t^C=P(u_0,u_1+C)-Ct, \quad C\in\R.
\eeq 
Here, for any $u,v\in\PSH\Omega$, the {\it rooftop envelope} $P(u,v)$ is the largest psh minorant of $\min\{u,v\}$ in $\Omega$ \cite{RWN}, 
$$P(u,v):=\sup\{ w\in \PSH\Omega \,:\, w\leq \min(u,v)\}.$$
More generally, given a measurable function $h$ which is locally bounded from above on $\Omega$, its {\it (lower) psh envelope} is defined as
$$P(h)=\sup \{v\in\PSH\Omega:\: v\le h \}^*,$$ where $^*$ denotes the u.s.c. regularization. In other worlds, $P(h)$ is the largest psh function 
which does not exceed $h$ quasi-everywhere in $\Omega$ (that is, outside a pluripolar set). Note that when $h$ is u.s.c., there is no need to take the upper regularization in the definition of the envelope.
\smallskip

Similarly, we define the {\it (lower) continuous psh envelope} 
\[
P^c(h):=\sup\{w\in\PSHc\Omega:\: w\le h\}.
\]
Note that the function $P^c(h)$ is l.s.c. as the upper envelope of continuous functions.
\smallskip

It follows immediately from the definition that the map $h\mapsto P(h)$ is superlinear, i.e., it satisfies $P(th)=tP(h)$ for all $t\ge 0$, and $P(h+g)\ge P(h)+P(g)$. The same holds for $h\mapsto P^c(h)$.

\section{Continuity of envelopes} 
In this section, we will review continuity properties of the rooftop envelopes $P(u,v)$ and, more generally, of the envelopes $P(h)$, assuming $u, v$, $h$ to be continuous in $\overline\Omega$. For a class of domains containing the B-regular ones, continuity can be established using the notion of Jensen measures and Edwards’ theorem, a technique whose application to plurisubharmonic functions was initiated in \cite{Si} and later further developed in subsequent works \cite{Pol91}, \cite{Pol93}, \cite{LS1}, \cite{W01}, \cite{Gog05}, \cite{DW05}, \cite{NW}, \cite{NW2}.
 After summarizing those results, we introduce a notion of Jensen pairs for $\omega$-psh functions on compact K\"ahler manifolds and establish an analog of Edwards' theorem for such functions, which gives an elementary proof of continuity of envelopes in this setting.


\subsection{Continuity when $h\in C(\overline\Omega)$}\label{ssec:Jensen}

Let $\Omega\subset\Cn$ be a bounded domain and $h\in C(\overline D)$. 
It was shown by Walsh that if $h$ is continuous in $\Omega$, then the set of discontinuities of $P(h)$ cannot have compact closure in $\Omega$ \cite[Theorem 2]{W}; furthermore, if $\Omega$ is B-regular, then $P(h)\in\PSHc{\Omega}$ \cite[Theorem 3]{W}. For more general classes of domains, as well as for some other applications, we adopt here an approach based on Jensen measures, following \cite{Si}, \cite{W01} (among others).

A positive measure $\mu$ on $\overline\Omega$ is called a {\it Jensen measure with barycenter at $z\in\overline\Omega$ for $\PSHc\Omega$} if 
\[u(z)\le\int_{\overline\Omega} u\,d\mu\quad \forall u\in\PSHc\Omega.\] 
The collection of all such measures is denoted by $\J_z^c$.

We mention that $\Omega$ is hyperconvex if and only if, for every $z\in\partial\Omega$, any measure $\mu\in\J_z^c$ is supported on $\partial\Omega$ \cite{CCW}, and $\Omega$ is B-regular if and only if, for every $z\in\partial\Omega$,  $\J_z^c=\{\delta_z\}$ \cite{Si}.

Similarly, $\J_z$ will be the collection of {\it Jensen measures $\mu$ with barycentre $z\in\overline\Omega$ for $\PSH\Omega$}, so that
\[
u^*(z)\le\int_{\overline\Omega} u^*d\mu\quad \forall u\in\PSH\Omega,\ \sup_\Omega u<\infty,
\] where $u^*$ is the u.s.c extension of $u$ to $\overline\Omega$.

By definition, $\J_z \subseteq \J_z^c$.
\medskip

 Since both $\PSH\Omega$ and $\PSHc\Omega$ are convex cones of u.s.c. functions and they contain all constants, it follows from Edwards' theorem \cite{Edw} (see \cite[Theorem 2.1 and Corollary 2.2]{W01} and \cite[Proposition 2.1]{DW05}) that, if $h$ is lower semicontinous on $\overline\Omega$, then 
\beq\label{Edwpsh}
P(h)(z)=E(h):=\inf\left\{\int_{\overline\Omega}h\,d\mu:\: \mu\in\J_z\right\},\ z\in\Omega,
\eeq

and
\beq\label{Edwpshc}
P^c(h)(z)=E^c(h):=\inf\left\{\int_{\overline\Omega}h\,d\mu:\: \mu\in\J_z^c\right\},\ z\in\overline\Omega.
\eeq
\smallskip
Observe that the latter equality makes sense because $P^c(h)$ is well defined up to the boundary.

When $\J_z=\J_z^c$ for all $z\in\Omega$, we can then infer that $P^c(h)=P(h)$ for any $h\in C(\overline \Omega)$. Indeed, the envelope $P(h)$ is plurisubharmonic and in particular u.s.c. in $\Omega$, while $P^c(h)$ is l.s.c.

It was shown in \cite{W01},\cite{Gog05} that that the implication $h\in C(\overline \Omega)\Rightarrow P(h)\in C(\overline \Omega) $ does not hold for all bounded (even hyperconvex) domains $\Omega$; while it does for domains with the following {\sl approximation property}: for any bounded function $u\in\PSH\Omega$, there
exists a sequence $u_j\in\PSHc\Omega$ decreasing to $u^*$ on $\overline\Omega$. The latter condition holds for all $B$-regular domains \cite[Theorem 4.1]{W01}; an example of a non-B-regular domain with that property is the unit bidisk \cite[Theorem 4.11]{W01}.
Also, the condition $\J_z=\J_z^c$ $\forall z\in\Omega$  is equivalent to a {\sl weak approximation property} \cite[Theorem~3.1]{DW05}, where the monotonicity of the sequence $u_j\in \PSHc{\Omega}$ is replaced by its uniform upper boundedness, pointwise convergence to $u$ on $\Omega$, and the condition $\limsup u_j\le u^*$ on $\partial\Omega$.




\medskip

We summarize this as
\begin{theorem}\label{thm:cutb} \cite{W01}, \cite{DW05}, \cite{Gog05} 
    Let $\Omega$ be a bounded domain with the weak approximation property (for example, B-regular) and $h\in C(\overline\Omega)$. Then $P(h)=P^c(h)\in C(\overline\Omega)$. In particular, $P(u,v)\in \PSHc\Omega$ for any $u,v\in\PSHc\Omega$.
\end{theorem}

\begin{remark}\label{rem:analyt disks}
In \cite{Pol91} (see also \cite[Proposition 2.1 and Theorem 2.2]{LS1}), a slightly different approach was developed. Namely,  it was shown that one can restrict to the Jensen measures given by closed analytic disks $f:\overline\D\to \Omega$ where $f$ is a holomorphic function such that $f(0)=z$; the collection of such maps is denoted by $A_z$, $z\in \Omega$. Then the measure $\mu_f$ on $\Omega$, $f\in A_z$, given by 
\[
\int_\Omega h\,d\mu_f:=\int_{\partial\D} f^*h\,d\sigma, \quad
h\in C(\overline\Omega),
\]
belongs to $\J_z$; here $\sigma$ is the normalized Lebesgue measure on $\partial\D$, and by monotone convergence theorem for integrals, $\mu_f$ can be extended to semicontinuous functions. By the main result of \cite{Pol91} (see also \cite[Proposition 2.1 and Theorem 2.2]{LS1}),
\beq\label{eq:Poletsky}
P(h)(z)=\inf_{f\in A_z} \int_\Omega h\,d\mu_f
\eeq
for any $h$ u.s.c. on $\Omega$.
\end{remark}

\subsection{Envelopes on compact manifolds}
Let $(X,\omega)$ be a compact K\"ahler manifold of complex dimension $n$ and let $\PSH{X, \omega}$ be the collection of all $\omega$-psh functions. We also set $\PSHc{X, \omega}:= \PSH{X, \omega} \cap C(X)$. We then let $P_\omega, P^c_\omega$ denote the corresponding $\omega$-psh envelopes.

In contrast to the local case, $\PSH{X, \omega}$ is not a cone, so Edwards' theorem cannot be applied directly to this case. One can, however, modify it by replacing Jensen measures by {\sl Jensen pairs}. 

We say that a pair $(\mu,b)$, $\mu$ being positive measure on $X$ and $b\in\R$,  is a {\it Jensen pair} with barycenter at $x\in X$, if for all $ u\in\PSH{X, \omega}$ 
\[
u(x)\le\int_X u\,d\mu +b.
\]
The collection of all such pairs is denoted by $\JP_{\omega,x}$. We emphasize that both $\mu$ and $b$ depend on $x\in X$.

\begin{example}\label{ex:Magnusson}
Examples of Jensen pairs are given through closed analytic disks: fix $x\in X$ and let $f:\overline\D\to X$ be a holomorphic map such that $f(0)=x$. We then consider a positive measure $\mu$ such that $f^*\mu$ is the normalized Lebesgue measure on the circle $\partial \D$ and we denote by $R_{f^*\omega}$ the potential of $f^*\omega$ on $\D$ (i.e. $f^*\omega=dd^c R_{f^*\omega}$) normalized such that $R_{f^*\omega} |_{\partial \D}=0$. We then claim that $(\mu, R_{f^*\omega}(0))$ is a Jensen pair. Indeed, given that $R_{f^*\omega}+f^*u$ is psh on $\D$, the mean value inequality gives
 $$(R_{f^*\omega}+f^*u)(0) \leq \frac{1}{2\pi} \int_{\partial \D} (R_{f^*\omega}+f^*u) \, d\sigma = \frac{1}{2\pi} \int_{\partial \D} f^*u \, d\sigma =\int_X u \, d\mu, $$
 that is equivalent to
 $$ u(x) \leq \int_X u \, d\mu -R_{f^*\omega}(0).$$
\end{example}

\smallskip
Given a bounded function $h$ on $X$, denote 
\[E_\omega(h)(x):=
\inf\left\{ \int_X h\,d\mu +b:\: (\mu,b)\in\JP_{\omega,x}\right\}.
\]
It follows directly from the definition that $P_\omega(h)\le E_\omega(h)$. Indeed, since $h\ge P_\omega(h)$, we have $\int_X h\,d\mu +b\ge \int_X P_\omega(h)\,d\mu +b\ge P_\omega(h)(x)$. One can then prove the following analog of Edwards' theorem.

\begin{theorem}\label{thm:Edwards}  If $h$ is a l.s.c function on $X$, then $P_\omega(h)=E_\omega(h)$.
\end{theorem}
\begin{proof}
We first assume $h\in C(X)$. Note that, the map $h\mapsto P_\omega(h)$ is concave. Indeed for any $t\in (0,1)$, we have
\[
P_\omega((1-t)h_0+t\,h_1)\ge P_{(1-t)\omega}((1-t)h_0)+P_{t\omega}(t\,h_1)=(1-t)P_\omega(h_0)+tP_\omega(h_1).
\]
This means that, for any fixed $x\in X$ and any $h$, the function $p(\tau):=P_\omega(\tau h)(x)$ is concave in $\tau\in \R$ and, therefore, there exists an affine function $r(\tau):=a\tau +b\ge p(\tau)$ for all $\tau\in\R$ and such that $r(1)=p(1)$ (a supporting line to the undergraph of $p$ at $1$). We emphasize here that both $a$ and $b$ depend on the point $x\in X$.\\
The linear function $r(\tau)-b$ defines a linear functional $l$
on the $1$-dimensional subspace $M_h=\{\tau h:\: \tau\in\R\}$ of $C(X)$ such that $l(\tau h) \ge P_\omega(\tau h)(x)-b$, for any $\tau \in \R$ and $l(h)=P_\omega(h)(x)-b$. 

By the convex version of the Hahn-Banach theorem (see, for example, \cite[Theorems 1.33 and 1.38]{BP}),
$l$ extends to a linear functional $L$ on $C(X)$ such that $L(g)\ge P_\omega(g)(x)-b$ for all $g\in C(X)$. Moreover, by Riesz' representation theorem, $L(g)=\int_X g\,d\mu$ for a real measure $\mu$ on $X$. \\
This gives 
\begin{equation}\label{ineq-Jensen pairs}
P_\omega(g)(x) \le \int_X g\,d\mu+b \quad \forall g\in C(X),
\end{equation}
and 
\begin{equation}\label{eq-Jensen pairs}
P_\omega(h)(x)= \int_X h\,d\mu+b.
\end{equation}
In addition, since $P_\omega(g)\ge 0$ for any $g\ge 0$, we have 
$\int_X g\,d\mu\ge -b$ for all $g\ge 0$. Applying this to $\tau g$ with $\tau\to +\infty$, we get $\int_X g\,d\mu\ge 0$, so $\mu$ is a positive measure.

Now, given any $u\in \PSH {X,\omega}$ there exists a sequence $g_j\in C(X)$ decreasing to $u$ \cite{BlKol} (we can even take them to be $\omega$-psh, however we do not need this here). Thus, using \eqref{ineq-Jensen pairs} and monotone convergence theorem we get
\[
u(x)=P_\omega(u)(x)\le\lim_{j\to\infty}P_\omega(g_j)(x)
\le \lim_{j\to\infty} \int_X g_j\,d\mu+b =\int_X u\,d\mu+b.
\]
This means that $(\mu,b)\in\JP_{\omega,x}$. The conclusion follows from  \eqref{eq-Jensen pairs}.

To conclude, we use an approximation argument from \cite[Theorem 2.1]{W01}, and prove that the equality still holds for all l.s.c. functions $h$. 
\end{proof}

\smallskip

Next, we give a simple proof, which uses Jensen measures, of a result established in \cite{GLZ} using bounds for the Laplacians of smooth approximations (obtained in \cite{Berm}).

\begin{theorem}\label{thm:omega_cont} 
If $h\in C(X)$, then $P_\omega(h)\in C(X)$.
\end{theorem}

\begin{proof}
    As in the local case, one can also consider $\omega$-psh envelopes $P_\omega^c$ with respect to continuous $\omega$-psh functions, corresponding Jensen pairs $\JP_{\omega,x}^c$, and the lower envelope $E_\omega^c(h)$. By repeating the arguments of the proof of Theorem~\ref{thm:Edwards} verbatim, we get the equality $P_\omega^c(h)=E_\omega^c(h)$ for any $h\in C(X)$.

\medskip 
We observe that the collections of the Jensen pairs for all $\omega$-psh functions and for the continuous ones {\sl always} coincide. Indeed, for every $x\in X$ we clearly have $\JP_{\omega,x}\subseteq \JP_{\omega,x}^c$. \\
For the reverse, let $(\mu,b)\in \JP_{\omega,x}^c$. By \cite{BlKol}, any $u\in\mathcal{PSH}(X, \omega)$ is the limit of a decreasing sequence of $u_j\in\mathcal{PSH}^c(X, \omega)$. By monotone convergence theorem we have
\[
u(x)=\lim_{j\to\infty}u_j(x)\le \lim_{j\to\infty} \int_X u_j\,d\mu+b =\int_X u\,d\mu+b,
\]
so $(\mu,b)\in \JP_{\omega,x}$. This implies $E_\omega(h)=E_\omega^c(h)$ and thus, $P_\omega^c(h)=P_\omega(h)$ for any $h\in C(X)$. The conclusion follows since $P_\omega^c(h)$ is l.s.c. and $P_\omega(h)$ is u.s.c.
\end{proof}

\begin{remark}
    In \cite{Mag11, Mag12}, Poletsky's theory of analytic disks, see Remark~\ref{rem:analyt disks}, was applied to $\omega$-psh functions on compact K\"ahler manifolds. Namely, it was shown that, given an u.s.c. function $h$, the envelope 
 $P_\omega(h)$ can be computed by taking the minimum over the Jensen pairs generated by closed analytic disks $f:\overline\D\to X$ (see Example~\ref{ex:Magnusson}). More precisely, it was shown that
    \[
    P_\omega(h)(x)=\inf\left\{\frac1{2\pi}\int_{\partial\D} h\circ f\,d\sigma -R_{f^*\omega}(0):\: f(0)=x\right\}.
    \]
\end{remark}

\section{Continuity of geodesics on domains}
 Let $\Omega$ be a bounded domain. We recall that the geodesic $u_t$ between two psh functions $u_0$ and $u_1$ is defined as $$u_t(z)=u_{\log|\zeta|}=\wu(z,\zeta)$$ where $\widehat u:=\sup\{v\in W(u_0,u_1)\}$ (see Section \ref{sec: pre}). We start with the following simple observation.


    \begin{proposition}\label{prop:env_geod}
    Let $u_0,u_1\in\PSH\Omega\cap L^\infty(\Omega)$. Then the function $\wu$ defined above is the largest solution of \eqref{MA_geo} and it equals the psh envelope in $\Omega\times\cA$ of the function 
    \beq\label{eq:Psi}
\Psi_u(z,\zeta):= 
(1-\log|\zeta|)\,u_0(z)+\log|\zeta|\,u_1(z),\quad (z,\zeta)\in\Omega\times \overline\cA.
\eeq 
In other words,
    \beq\label{eq:env_geod}
    \wu =P_{\Omega\times\cA}(\Psi_u).
    \eeq
         \end{proposition}
 Here by the ``largest solution" we mean the upper envelope of all solutions.
\begin{proof}
For the first statement we observe that since $\hat{u}$ is a solution, we have $\hat{u} \leq \sup\{w\in \PSH{D} \, :\, w\;{\rm{is\; a\; solution\; of}}\; \eqref{MA_geo}\}$. On the other hand $$\{ {\rm{ solutions\, of}}\, \eqref{MA_geo}\}\subseteq \{ v \, :\, v\in W(u_0,u_1)\},$$
that gives the reverse inequality.\\
   For the second statement we observe that $\wu\le P_{\Omega\times\cA}(\Psi_u)$ since $\wu\le \Psi_u$, and then the equality results from the fact that $\Psi_u=u_j$ on $\Omega \times \cA_j$ and thus, $P_{\Omega\times\cA}(\Psi_u)\in W(u_0,u_1)$. 
\end{proof}
\smallskip

The geodesic is Lipschitz in $t$ for any bounded $u_0,u_1$.  Indeed by using argumetns from \cite{Bern}, we get

\begin{proposition}\label{prop:Lip in t}
    Let $u_0,u_1\in \PSH{\Omega}$ be bounded functions. Then the geodesic $u_t$ is Lipschitz in $t\in {[0,1]}$, uniformly in $\Omega$: 
    $$|u_t-u_s|\le M|t-s|,$$ 
    where $M=\|u_0-u_1\|_\infty$.
\end{proposition}

\begin{proof} As we already observed
$$V_t:=\max\{u_0-M\,t,u_1-M\,(1-t)\},$$ is a subgeodesic (see \eqref{VtM}). 
In particular, the inequality $V_t\le u_t$ implies that for any $z\in \Omega$
\[
\frac{u_t(z)-u_0(z)}t\ge \frac{V_t(z)-u_0(z)}t =\frac{V_t(z)-V_0(z)}t,
\]
so
\[
\dot{u}_0(z)=\lim_{t\to 0^+}\frac{u_t(z)-u_0(z)}t\ge 
\lim_{t\to 0^+}\frac{V_t(z)-V_0(z)}t \ge-M.
\]
Similarly, $\dot{u}_1(z)\le M$. Since  $u_t$ is convex in $t$, the ratio
\[
\frac{u_t(z)-u_s(z)}{t-s}, \quad 0<s<t<1,
\]
is increasing both in $s$ and $t$, so 
\[
-M\le \dot{u}_0\le \frac{u_t(z)-u_s(z)}{t-s} 
\le \dot{u}_t\le \dot{u}_1\le M, \quad 0<s<t<1.
\]
The conclusion then follows.
\end{proof}

In what follows we study continuity of the geodesics (up to the boundary). We will first treat the case when the boundary values on $\partial \Omega \times \mathcal{A}$ are zero, and then we will consider the case of continuous boundary values.

\subsection{Continuity: case of zero boundary values}
We consider the following Dirichlet problem 
\begin{equation}\label{eq: Dirichlet with boundary}
\begin{cases}
(dd^c  w)^{n+1}=0\qquad  \textit{on}\; D:=\Omega\times\cA\subset\C^{n+1}\\
 w=u_j \hspace{2cm} \textit{on}\;  \Omega\times\cA_j, j=0,1\\
 w=0 \hspace{2.2cm}  \textit{on}\; \partial\Omega\times\cA.
\end{cases}
\end{equation}
We emphasize that, in contrast with the global case, there is no uniqueness in the Dirichlet problem \eqref{eq: Dirichlet with boundary} without specifying the boundary conditions on $\partial\Omega\times\cA$  as well. For example, the function $w(z,\zeta)=\log|\zeta|(\log|\zeta|-1)\in\PSHc{D}$ connects $u_0\equiv 0$ with $u_1\equiv 0$ and satisfies $(dd^c w)^{n+1}=0$ in $D$. 

In \cite[Theorem 2.4]{A} was shown that, if $\Omega$ is strongly pseudoconvex and $u_0,u_1\in \PSHc{\Omega}$ equal zero on $\partial\Omega$, then the function $\widehat u$ solves \eqref{eq: Dirichlet with boundary} and belongs to $C (\overline{D})$.
We observe that requiring $u_0=u_1=0$ on $\partial\Omega$ forces $\Omega$ to be hyperconvex (cf. the definition in Section \ref{sec: pre} and take $\rho=u_j$, $j=0,1$). We reprove and slightly generalize \cite[Theorem 2.4]{A} in the case $\Omega$ is merely hyperconvex (not necessarily strongly pseudoconvex). Note that the regularity statement does not automatically follow from the general result of Bedford and Taylor because $\Omega\times\cA$ is neither smooth nor strongly pseudoconvex.

 \begin{theorem}\label{thm:geod_cont}
 Let $\Omega$ be  a bounded hyperconvex domain and $u_0,u_1\in\PSHc\Omega$, $u_0=u_1=0$ on $\partial\Omega$. Then $\widehat u $ belongs to $\PSHc{D}$ and is the unique psh function  that satisfies \eqref{eq: Dirichlet with boundary}.
 \end{theorem}

\begin{proof}
 
We already know that $\widehat{u}\in \PSH D$ and that $(dd^c\widehat{u})^{n+1}=0$ in $D$. We need to show that $\widehat u|_{\partial D}=\Psi_u$ \eqref{eq:Psi}; note that  $\Psi_u\in C(\overline D)$.

Denote $\wz=(z,\zeta)$ and recall that 
$\wu(\wz)=\widehat u(z,|\zeta|)=
u_{\log|\zeta|}(z)=u_t(z)$ is the geodesic for $u_0,u_1$.
Let $M:=\|u_0-u_1\|_\infty$, then 
\[
v_t(z):=\max\{u_0-Mt,u_1-M(1-t),u_0+u_1\}
\]
is a subgeodesic since, by assumtpion and the maximum principle, we have $u_0, u_1\leq 0$. 
Therefore, 
$$v_t\le u_t\le (1-t)u_0 + t u_1,$$
which shows that $\widehat u(\wz)$ converges uniformly to $\Psi_u$ when $\wz\to\partial D$. We extend $\wu$ to $\partial D$ by setting $\wu=\Psi_u$ there.

Continuity of $\wu$ follows then from \cite[Theorem 2]{W}, For the sake of completeness, we present here the arguments inspired by \cite{GZ12,GZ}.

Since  $\lim_{\wz\to \xi\in\partial D} \wu(\widehat z)=\Psi_u(\xi)$  and $\Psi_u$ is uniformly continuous on $\partial D$ (a continuous function on a compact set is uniformly continuous there), then $\forall\varepsilon>0$ $\exists\delta>0$ such that 
\beq\label{eq:unifcont}
|\wu(\wz)-\Psi_u(\xi)|\le\varepsilon\quad \forall(\wz,\xi)\in D\times \partial D:\ |\wz-\xi|\le\delta.
\eeq
Choose $a\in\C^{n+1}$ such that $|a|<\delta$ and denote $D_a=D-a$ (i.e., we are translating the domain). 
By (\ref{eq:unifcont}), for any $\wz\in D_a\cap\partial D$ we have $\hat{z}+a \in {D}$ and $$\wu(\wz+a)\ge\Psi_u(\wz)-\varepsilon=\wu(\wz)-\varepsilon,$$
while for any $\wz\in D\cap \partial D_a$ we have $\hat{z}+a \in \partial{D}$ and
$$\wu(\wz+a)=\Psi_u(\wz+a)\ge \wu(\wz)-\varepsilon.$$ 
This shows that $\wu(\wz)-\varepsilon\le \wu(\wz+a)$ on $\partial(D\cap D_a)$ and thus in $D\cap D_a$, since $\wu(z+a)$ is maximal there.
Since $a$ is an arbitrary point with $|a|\le\delta$, this shows that $\wu$ is l.s.c.\\
Its upper semicontinuity as a psh function implies $\wu\in C(D)$. In view of the uniform convergence of $\widehat u(\wz)$ to $\Psi_u$ when $\wz\to\partial D$ (established above), we get $\wu\in C(\overline D)$.
\end{proof}

\subsection{Continuity: arbitrary continuous boundary values}

Now we drop the condition $u_0=u_1=0$ on $\partial\Omega$.
Let us consider  the {\it ``partial" Dirichlet problem} 
\beq\label{eq:PDP}
\begin{cases}
(dd^c w)^{n+1}=0\quad  {\rm on\ } D=\Omega\times\cA\\
 w|_{\Omega\times \cA_j}=u_j\hspace{1.2cm} j=0,1.
\end{cases}
\eeq

This problem goes back to \cite{Brem} where the author shows that, given a pseudoconvex domain $D$ of the form $D=\{V<0\}$ for a plurisubharmonic function $V$ on a neighborhood of $\overline D$ and a continuous function $\phi$ on the \v{S}ilov boundary $\partial^S D$ of $D$, the upper envelope $\widehat w$ of all $w\in\PSH{D}$ not exceeding $\phi$ on $\partial^S D$ is plurisubharmonic in $D$ and attains the boundary values $\phi$ on $\partial^S D$.

In the case of interest, $D=\Omega\times\cA$ where $\Omega$ is a hyperconvex domain. It is not difficult to see that $\partial^S D\subseteq\overline\Omega\times\partial\cA$.
Our next result generalizes the above by considering the boundary values on the all $\partial D$.

\smallskip

For arbitrary bounded hyperconvex domains $D$, boundary values of $w\in\PSHc{D}$ can be described in terms of Jensen measures (see Section~\ref{ssec:Jensen}) as follows:

\begin{theorem}\label{thm:range}
\cite[Theorem 3.5]{W01} For any bounded hyperconvex domain $D$, a function $\phi\in C(\partial D)$ can be extended to a function in $\PSHc{D}$ if and only if
\beq\label{eq:extension}
\phi(z)=\inf\left\{\int_{\partial D}\phi\,d\mu:\ \mu\in\J_z^c\right\} \quad\forall z\in\partial D.
\eeq
\end{theorem}


If $\Omega$ is a B-regular domain, then any $\phi\in C(\partial\Omega)$ extends to a function in $\PSHc{\Omega}$. However, the product domain $D=\Omega\times\cA$ is never $B$-regular (while it is hyperconvex), even if $\Omega$ is. 
By using Theorem~\ref{thm:range}, we get the following explicit characterization of functions $\phi\in C(\partial D)$ that extend to $\PSHc{D}$.

\begin{proposition}\label{prop:bv}
Let $\Omega$ be a $B$-regular domain and $D=\Omega\times\cA$. Let $u_0,u_1\in\PSHc\Omega$ and let $\phi\in C(\partial D)$ equal to $u_j$ on $\Omega\times\cA_j$, $j=0,1$. Then $\phi$ can be extended to a function in $\PSHc{D}$ if and only if for any $p\in\partial\Omega$, the function $\phi_p(\zeta):=\phi(p,\zeta)$ is subharmonic in $\cA$. 
\end{proposition}
We remark that when we require, in addition, $\phi$ to be independent of $\arg\zeta$, subharmonicity of $\phi\in\SH\cA$ becomes convexity of the function $ (0,1) \ni t\rightarrow \phi_p(e^t)$  for any $p\in\partial\Omega$. 

\begin{proof} If $\phi=\Phi$ on $\partial D$ for some $\Phi\in \PSHc{D}$, then $\phi_p$ is the limit of $\Phi_{p_k}\in\SHc\cA$ for $p_k\in\Omega$ converging to $p$ and thus is subharmonic in $\cA$. 
\smallskip

For the viceversa, let $\phi\in C(\partial D)$ be such that $\phi_p$ is subharmonic in $\cA$ for any $p\in\partial\Omega$. By Theorem \ref{thm:range} it suffices to show that \eqref{eq:extension} is satisfied.

First, we will treat the points $\widehat z=(p,\zeta)\in \partial\Omega\times\cA$. Since $\Omega$ is $B$-regular, \cite[Theorem 2.1]{Si} ensures the existence of $\rho_p\in\PSHc{\Omega}$ satisfying $\rho_p(z)<0$ for all $z\in\overline\Omega$ different from $p$, and $\rho_p(p)=0$. Extending it as $V_p(z,\zeta)=\rho_p(z)$ to $D$, we get $V_p\in\PSHc{D}$ satisfying $V_p\le 0$ everywhere and $V_p=0$ on $\{p\}\times\overline\cA$. Then, for any $\widehat z=(p,\zeta)\in \partial\Omega\times\cA$ and $\widehat\mu\in\J_{\widehat z}^c$, 
\[
0=V_p(\widehat z)\le \int_{\overline{D}}V_p \,d\widehat\mu \leq 0.
\]
This means that $\int_{\overline{D}}V_p \,d\widehat\mu=0$ and, by definition of $V_p$, this implies that $\supp\,\widehat\mu\subset \{p\}\times\overline\cA$ for any $\widehat\mu\in\J_{\widehat z}^c$.\\
Therefore, for any $U\in\PSHc{D}$ and any $\widehat\mu\in\J_{\widehat z}^c$, we have 
\[
u_p(\zeta)=U(p,\zeta)\le\int_{{p}\times\overline{\cA}} U \,d\widehat\mu=\int_{\overline{\cA}}u_p\, (\pi_2)_\star(d\widehat\mu),\quad (p,\zeta)\in \partial\Omega\times\overline{\cA},
\]
where $\pi_2$ is the projection $\pi_2(z,\zeta)=\zeta$. 
Since any function in $\SHc\cA$ is of the form $u_p$, we obtain that $(\pi_2)_\star \widehat\mu$ is a Jensen measure for $\SHc\cA$, so 
$(\pi_2)_\star\J_{\widehat z}^c(D)\subset \J_{\zeta}^c(\cA)$. \\
Conversely, any $\mu\in\J_{\zeta}^c(\cA)$ defines a Jensen measure $\widehat\mu_p\in\J_{\widehat z}^c(D)$ by setting $\widehat \mu_p(E)=\mu(\pi_2 (E))$ supported on $\{p\}\times\overline\cA$ and such that $(\pi_2)_\star \widehat \mu_p=\mu$, which gives us 
$ \J_{\zeta}^c(\cA) \subseteq (\pi_2)_\star\J_{\widehat z}^c(D)$, hence equality.

Therefore, in the case $\wz=(p,\zeta)\in\partial\Omega\times\cA$, equation \eqref{eq:extension} is equivalent to 
\[
\phi_p(\zeta)=\phi(\wz)=\inf\left\{\int_{\partial \Omega\times \overline\cA}
\phi\,d\widehat{\mu}:\ \widehat{\mu}\in\J_{\wz}^c(D) \right\} =\inf\left\{\int_{\overline\cA}\phi_p\,d\mu:\ \mu\in\J_\zeta^c(\cA)
\right\},
\]
which holds because $\phi_p\in\SHc\cA$.
This proves (\ref{eq:extension}) on $\partial\Omega\times\cA$.
\smallskip

Now we check it on $\overline\Omega\times\cA_0$. Take $V_0(z,\zeta)=-\log|\zeta|\in\PSHc{D}$ (actually, harmonic since $\zeta\in \C\setminus\{0\}$). Then, for any $\mu\in\J_{\wz}^c$, $\wz\in\overline\Omega\times\cA_0$,
\[
0=V_0(\wz)\le \int_{\overline D}V_0\,d\mu\le 0
\]
and, therefore, $\supp\,\mu\subset \overline\Omega\times\cA_0$. The same arguments as above give us $(\pi_1)_\star\J_{\widehat z}^c(D)= \J_{z}^c(\Omega)$ for all $\wz\in\overline\Omega\times\cA_0$.

Since $u_0\in\PSHc\Omega$, we have 
\[
u_0(z) = \inf\left\{\int_{\overline\Omega}u_0\,d\mu:\ \mu\in\J_{ z}^c\right\} \quad\forall z\in\overline\Omega
\]
(to get the equality, we can just take $\mu=\delta_z$) and so,
\beq\label{eq:infonA0}
\phi(\wz) =u_0(z)= \inf\left\{\int_{\overline\Omega \times \cA_0 } \pi_1^\star u_0\,d\widehat\mu:\ \widehat\mu\in\J_{\widehat z}^c\right\} \quad\forall \wz\in\overline\Omega\times\cA_0.
\eeq

Similarly for the case $\hat{z}\in \overline\Omega\times\cA_1$, we can use the function $v_1(z,\zeta)=\log|\zeta|-1$ which is negative in $D$, and get 
\beq\label{eq:infonA1}
\phi(\wz) = \inf\left\{\int_{\overline\Omega \times \cA_1} \pi_1^\star u_1\,d\widehat\mu:\ \widehat\mu\in\J_{\wz}^c\right\} \quad\forall \wz\in\overline\Omega\times\cA_1.
\eeq
Then (\ref{eq:infonA0}) and (\ref{eq:infonA1}) give us (\ref{eq:extension}) for all $\widehat z\in \overline\Omega\times\partial\cA$. This completes the proof.
\end{proof}

\medskip


In Proposition \ref{prop:env_geod} we have shown that $\widehat{u}$ is the largest solution of the partial Dirichlet problem. It then follows that it has the largest boundary values on $\partial\Omega\times\cA$, which gives us

\begin{theorem}\label{thm:geo_bv}
 Let $\Omega$ be a $B$-regular domain, $D=\Omega\times\cA$, and let $u_0,u_1\in\PSHc\Omega$. Then the largest solution  
  $\wu$ of the partial Dirichlet problem (\ref{eq:PDP}) belongs to $C(\overline{D})$, has the boundary values $\wu(z,\zeta)=u_j(z)$ for $(z,\zeta)\in\overline\Omega\times\cA_j$, $j=0,1$, and $\wu(z,\zeta)=\Psi_u(z,\zeta)$ for
  $(z,\zeta)\in\partial\Omega\times\cA$, where
\beq\label{eq:lingeod0}
\Psi_u(z,\zeta)=(1-\log|\zeta|)\,u_0(z)+\log|\zeta|\,u_1(z),\quad  (z,\zeta)\in\overline D.
\eeq 
In other words, the geodesic $u_t$ satisfies 
\beq\label{lingeod}u_t(z)=(1-t)u_0(z)+t\,u_1(z), \quad z\in\partial\Omega,\ 0<t<1.
\eeq
\end{theorem}

\begin{proof}
The fact that $\wu$ is the largest solution of \eqref{eq:PDP} has been already observed. Since for any $p\in \Omega$ the function $\Psi_u(p,\cdot)$ is subharmonic in $\cA$, Proposition \ref{prop:bv} ensures that there exists $v\in \PSHc{D}$ such that $v|_{\partial D}=\Psi_u$. Such a function $v$ is then a candidate for the envelope $\wu$. On the other side, Proposition \ref{prop:env_geod} implies that
\begin{equation}\label{eq: sandwich geo}
v\leq \wu\leq \Psi_u,
\end{equation}
hence
\begin{equation}\label{eq: sandwich geo1}
\Psi_u=v|_{\partial D}\leq \wu|_{\partial D} \leq \Psi_u.
\end{equation}
Finally, the continuity up to the boundary follows from \eqref{eq: sandwich geo} and \eqref{eq: sandwich geo1}, and then the continuity in $\Omega\times\cA$ follows from (\ref{eq:env_geod}) and \cite[Theorem 2]{W} (see also the end of the proof of Theorem~\ref{thm:geod_cont}).
\end{proof}
\begin{remark}
    Note that the equality $\wu=\Psi_u$ cannot hold on an open subset of $D$ unless $u_0-u_1$ is constant there, otherwise it contradicts plurisubharmonicity of $\wu$.
\end{remark}

\subsection{Darvas-Rubinstein formula}

In \cite{DR}, Darvas and Rubinstein get useful relations between geodesics joining bounded $\omega$-psh functions on a compact K\"ahler manifold $(X,\omega)$ and $\omega$-psh envelopes. In the local setting, the analogous identities hold:

\begin{theorem}\label{thm:DRrepresentation}
Let $\Omega$ be a bounded domain. Let $u_0,u_1\in\PSH\Omega\cap L^\infty(\Omega)$ and $u_t$ be the geodesic joining them. Then 
\beq\label{eq:DRrepresentation}
u_t(z)=\sup_{C\in \R} \{P^C(z)+Ct \}, \ 0<t<1,\ z\in\Omega,
\eeq
where $P^C:=P(u_0,u_1-C)$. 
Conversely,
\beq\label{eq:DRenv}
P^C(z)=\inf_{t\in (0,1)}\{u_t(z)-Ct \}=:u^C(z),\quad C\in\R,\ z\in\Omega.
\eeq
In particular, 
\beq\label{eq:DRenv1}
P(u_0,u_1)=\inf_{t\in (0,1)} u_t.
\eeq
Moreover, if $\Omega$ is B-regular and $u_0,u_1\in\PSHc\Omega$, then \eqref{eq:DRrepresentation} and \eqref{eq:DRenv} hold on $\partial\Omega$ as well.
\end{theorem}

\begin{proof} We start establishing (\ref{eq:DRenv}), following \cite{Da14a} (see also \cite[Prop. 5.1]{Da14b}).

First, we note that by Kiselman's minimum principle \cite{Kis}, $ u^C\in\PSH\Omega$.

Since by definition $v_t^C:= P^C+Ct$ is a subgeodesic, we have $v_t^C\le u_t$. Thus 
$ P^C=v_t^C-Ct\le u_t-Ct$ for any $t\in(0,1)$. Taking the infimum over $t$ we arrive at $P^C\le u^C$.

For the other inequality we observe that by definition 
$$ u^C\le \limsup_{t\to 0} (u_t-Ct)\le u_0\quad \textit{and} \quad u^C\le \limsup_{t\to 1} (u_t-Ct)\le u_1-C$$ which gives that $u^C$ is a candidate in the envelope $P^C$, i.e. $u^C\le P^C$. This proves \eqref{eq:DRenv}.

We now observe that for any $z\in \Omega$
\[
u^C(z)= -\sup_{t\in (0,1)}\{Ct-u_t(z)\}=-\cL(u_t(z)),
\]
where with $\cL(u_t(z))$ we denote the Legendre transform of the function $(0,1)\ni t\to u_t(z)$. 
Therefore, applying the Legendre transform to \eqref{eq:DRenv} we get that for any $z\in \Omega$,
$$ u_t(z)=\cL (-P^C(z))= \sup_{C\in \R} \{Ct+P^C(z)\},$$
that is \eqref{eq:DRrepresentation}.

To prove the last statement, we note that Proposition~\ref{prop:env on boundary} below gives
\beq\label{eq:infmin}
P(u_0,u_1-C)(z)=\min\{u_0(z),u_1(z)-C\}, \ \forall z\in\partial\Omega.
\eeq
Since 
$$\inf_t \{(1-t) u_0+ t(u_1-C)\}= \min (u_0, u_1-C)$$ and $u_t(z)=(1-t)u_0(z)+t\,u_1(z)$ for any $z\in\partial\Omega$ thanks to Theorem~\ref{thm:geo_bv},
we get (\ref{eq:DRrepresentation}) for $z\in \partial \Omega$ and, by the duality, \eqref{eq:DRenv}. 
\end{proof}

\begin{proposition}\label{prop:env on boundary}
  If $\Omega$ is B-regular and $h\in C(\overline\Omega)$, then $P(h)=h$ on $\partial\Omega$. In particular, if  $u,v\in\PSHc\Omega$, then 
   $P(u,v)=\min\{u,v\}$ on $\partial\Omega$.
\end{proposition}

\begin{proof} 
Since $\Omega$ is B-regular, there exists a function $w\in\PSHc\Omega$ equal to $h$ on $\partial\Omega$. Then, for any $\epsilon>0$, there exists a subdomain $\Omega'\Subset \Omega$ such that $w-\epsilon \le h$ in $\Omega\setminus \Omega'$. Let $\rho\in\PSHc\Omega$ be an exhaustion function for $\Omega$. Since $\sup\rho<0$ on $\Omega'$, we have $w-\epsilon+N\rho<h$ in $\Omega$ for $N>0$ sufficiently large. This gives us $P(h)\ge w-\epsilon+N\rho$ on $\Omega$ and thus, $P(h)\ge h-\epsilon$ on $\partial\Omega$.
As $P(h)\le h$, this completes the proof.
\end{proof}

\begin{remark} 
In the {\sl toric} setting, the description of geodesic joining toric ($\omega$-)psh functions $u_0, u_1$ is given in terms of Legendre transforms. We refer to \cite{Gu} for compact K\"ahler manifolds $(X,\omega)$ and to \cite{R17} for domains in $\Cn$. We emphasize that this representation is different to the one used in Theorem \ref{thm:DRrepresentation}. Indeed, while in (\ref{eq:DRenv}) $u_t(z)$ is considered as a convex function of $t$ with $z$ fixed, the representations from \cite{Gu, R17} transform a toric function $u$ into a convex one, $\check u(s)=u(e^{s_1},\ldots, e^{s_n})$, and consider their Legendre transforms $\tilde\cL[\check u]$ in $s\in\R^n$ with $t$ fixed. Doing this, one gets
    \[
    \tilde\cL[\check u_t] =(1-t)\tilde\cL[\check u_0] + t\tilde\cL[\check u_1],
    \]
    see \cite[Thm. 5.1]{R17}.
\end{remark}

\section{Smoothness}

Let $\Omega$ be a strongly pseudoconvex bounded domain of $\Cn$ and $u_0,u_1\in\PSHc{\Omega}$. By Theorem~\ref{thm:geo_bv},  $\wu$ solves the Dirichlet problem in $D=\Omega\times\cA$ with boundary values $\Psi_u$ defined by (\ref{eq:Psi}), and $\wu\in\PSHc{D}$. It is also shown in Proposition~\ref{prop:Lip in t} the $\wu$ is Lipschitz continuous in $\zeta\in\overline\cA$.\\
In this section we assume that $u_j$ have some additional smoothness. 

First, we treat the case when the functions $u_j$ are $\alpha$-H\"older continuous in $\overline\Omega$, and we conclude $\wu$ is $\frac\alpha2$-H\"older continuous in $z\in \overline{\Omega}$  (Theorem~\ref{thm:1/2 H}).

We next assume that $u_0,u_1\in C^{1,1}(\overline\Omega)$ with ``essentially" the same boundary values, and we show that $\wu$ is Lipschitz in $z\in\overline\Omega$.

Ideally, one would also like to prove that $\wu$ belongs locally to $C^{1,1}$ in $z\in\Omega$. At the moment, we can only treat the case of $\Omega=\B_n$.

\subsection{H\"older regularity}

We recall that, by a general Walsh type result of Bedford and Taylor \cite[Thm. 6.2]{BT76}, the solution  to the homogeneous Monge-Ampère equation in a strongly pseudoconvex domain with $\alpha$-H\"older continuous boundary values on is $\frac\alpha2$-H\"older continuous in $\overline\Omega$. One cannot apply it directly to our $D=\Omega\times\cA$ (which is not even B-regular). Nevertheless, the approach can be adapted to this case. (For the case of geodesics between convex functions with zero boundary values, see \cite{AD}.)

\begin{theorem}\label{thm:1/2 H}
 Let  $\Omega$ be strongly pseudoconvex. Then the following hold:
 \begin{itemize}
 \item[(i)] Assume $u_0,u_1\in\PSH\Omega$ are $\alpha$-H\"older continuous on $\overline\Omega$, $0<\alpha\le 1$. Then the geodesic $u_t$ is $\frac\alpha2$-H\"older on $\overline{\Omega}$, uniformly in $t\in [0,1]$.
 
 \item[(ii)] Assume $u_0,u_1\in C^{1,1}(\overline \Omega)$ and $u_0-u_1=const$ on $\partial \Omega$. Then $u_t$ is Lipschtiz on $\overline{\Omega}$, uniformly in $t\in [0,1]$.
 \end{itemize}
\end{theorem}
The proof of $(i)$ relies on an idea communicated to us by S. Dinew.

\begin{proof}
We start by proving the first statement.
    The conclusion will follow from the fact that $\wu$ has two H\"older continuous barriers. The upper barrier is the function 
    $$\Psi(\wz):=\Psi_u(z,\zeta)=(1-\log|\zeta|)\,u_0(z)+\log|\zeta|\,u_1(z),\quad \wz\in \overline D=\overline{\Omega\times\cA},$$ 
   satisfying 
    \beq\label{eq:hPsi}\Psi(\wz+\hat h)-\Psi(\wz)\le C_1|h|^\alpha\eeq  
    for all $\wz=(z,\zeta)\in \Cn\times\C$ and $\hat h=(h,0)\in \Cn\times\C$ such that $z,z+h   \in\overline{\Omega}$, where $C_1=C_1(u_0,u_1)$. Recall that 
    \beq\label{eq:hulb1}
    \wu(\wz)\le\Psi(\wz), \quad \wz\in D,
    \eeq
and by Theorem~\ref{thm:geo_bv},
    \beq\label{eq:hulb2} 
    \wu(\wz)=\Psi(\wz), \quad \wz\in \partial D.
    \eeq

 Constructing a lower barrier is more subtle. Consider, the subgeodesic 
    \beq\label{eq:subg V}
    V(\wz)=\max\{u_0(z)-M\log|\zeta|,u_1(z)-M\,(1-\log|\zeta|)\},
    \eeq
    where $M=\|u_0-u_1\|_\infty$. This
    coincides with $\wu$ on $\partial\Omega\times\cA$ only if $u_0-u_1= M$ on $\partial \Omega$; otherwise, one needs to find a subgeodesic with correct values on $\partial\Omega\times\cA$. 

We then construct a family of $\frac\alpha2$-H\"older continuous lower envelopes $V_{z_0}(\wz)$, $z_0\in\partial\Omega$, such that $V_{z_0}\le\wu$ on $D$, and $V_{z_0}=\wu$ on $(\Omega\times\partial\cA)\cup(\{z_0\}\times\cA)$.

Let $\eta_{z_0}\in\PSHc\Omega$ be the solution to the homogeneous Monge-Amp\`ere problem $(dd^c \eta)^n=0$ on $\Omega$ with the boundary values $$\eta_{z_0}(z)=-|z-z_0|^\alpha, \quad z\in\partial\Omega. $$ 
By the aforementioned \cite[Theorem 6.2]{BT76}, $\eta_{z_0}$ is $\frac\alpha2$-H\"older continuous on $\overline\Omega$.
Since, by maximality, $\eta_{z_0}(z)+|z-z_0|^\alpha=0$ on $\partial\Omega$, we have $\eta_{z_0}(z)+|z-z_0|^\alpha< 0$ in $\Omega$, so 
\beq\label{eq:h-eta}
\eta_{z_0}(z)\le -|z-z_0|^\alpha, \quad z\in\overline\Omega.
\eeq

Now, consider the function
    \beq\label{eq:vz0}
    v_{z_0}(\wz)= \Psi(\wz_0)+C_1\eta_{z_0}(z)\in\PSHc{D},
    \eeq
    where $\wz_0=(z_0,\zeta)$. When $z\in\partial\Omega$ and $\zeta\in\cA$,  by \eqref{eq:hPsi} we have
    \[
    v_{z_0}(\wz)=v_{z_0}(z,\zeta)= \Psi(z_0,\zeta)-C_1 |z-z_0|^\alpha\le \Psi(z,\zeta)=\Psi(\wz)
    \]
    and by \eqref{eq:hulb2}
    \[
    v_{z_0}(\wz_0)=v_{z_0}(z_0,\zeta)=\Psi(\wz_0).
    \]

        On each $\Omega\times\cA_j$, we have, by \eqref{eq:h-eta}, \eqref{eq:hPsi}, and \eqref{eq:hulb2},
    \[
    v_{z_0}(\wz)=u_j(z_0)+C_1\eta_{z_0}(z) \le
    u_j(z_0)-C_1|z-z_0|^\alpha\le u_j(z)=\Psi(\wz).
    \]
    Therefore, $v_{z_0}\le \Psi$ on $\partial D$ and hence it is a subgeodesic for $(u_0,u_1)$: $v_{z_0}\le\wu$.
    
By construction, the function 
    \beq\label{eq:lbarrier}
    V_{z_0}=\max\{V, v_{z_0}\}\in\PSHc{D}.
    \eeq
 satisfies $V_{z_0}\le\wu$. Moreover $V_{z_0}=\wu$ on $\Omega\times\partial\cA$ and on $\{z_0\}\times\cA$.  
    This is our lower barrier for $z_0\in\partial\Omega$, which is H\"older continuous with exponent $\frac\alpha2$:
\beq\label{eq:eta 1/2}
V_{z_0}(z+h)-V_{z_0}(z)\le C_2|h|^{\frac\alpha2},\quad \forall z,h: \ z,z+h\in\overline\Omega,
\eeq
where the constant $C_2=C_2(\alpha, \Omega)$ is chosen such that 
\beq\label{eq:1/2 and 1}
C_2|h|^{\frac\alpha2}\geq C_1|h|^\alpha, \quad |h|\le {\rm diam}(\Omega).\eeq 

Given any small $\hat h=(h,0)\in\Cn\times\C$, we define a domain
    \[
    D_{h}=\{\wz\in D:\: \wz+\hat h\in D\}
    \]
    and a function $w\in\PSHc{D_h}$, 
    \[
    w(\wz)=\wu(\wz+\hat h)- C_2| h|^{\frac\alpha2}.
    \]
   
    We claim that $w\le\wu$ on $D_{ h}$.
    Since $\wu$ is maximal, it suffices to prove this on $\partial D_{ h}$. We split $\partial D_{ h}$ into two parts: $\Gamma_1=\partial D_{ h}\cap\partial D$ and $\Gamma_2=\partial D_{ h}\cap D$.

    Take any $\wz_0\in\Gamma_1$. Then, by \eqref{eq:hPsi}, \eqref{eq:hulb1},  \eqref{eq:hulb2}, and \eqref{eq:1/2 and 1},
    \begin{eqnarray*}
        w(\wz_0) &=& \wu(\wz_0+\hat h)- C_2|h|^{\frac\alpha2} \le \Psi(\wz_0+\hat h) -C_2| h|^{\frac\alpha2} \\ 
        &\le & \Psi(\wz_0)+ C| h|^\alpha - C_2| h|^{\frac\alpha2}
        \le \Psi(\wz_0) = \wu(\wz_0),
    \end{eqnarray*}
    so $w\le\wu$ on $\Gamma_1$.

Let now $\wz_0\in\Gamma_2$. Then, by definition of $\Gamma_2$, $z_0+h\in \partial\Omega$ (because if $z_0+h\in\Omega$, then $\wz_0\in D_h$, and if $z_0+h\not\in\overline\Omega$, then $\wz_0\not\in \overline{D_h}$). Then we can consider the function $V_{z_0+h}$ instead of $V_{z_0}$. By the above $V_{z_0+h}\le \wu$ in $D$ and $V_{z_0+h}(\wz_0+h)=\wu(\wz_0+\hat h)$. Also,
 \[
        w(\wz_0) =\wu(\wz_0+\hat h) -C_2|h|^{\frac\alpha2} = V_{z_0+ h}(\wz_0+\hat h) -C_2|h|^{\frac\alpha2}
        \le  V_{z_0+h}(\wz_0)
        \le \wu(\wz_0),
\]    
 so  $w\le\wu$ on $\Gamma_2$ as well and thus, on the whole $\partial D_{ h}$, which proves the claim. Therefore, 
 \[
 \wu(\wz+\hat h)- \wu(\wz) \le C_2| h|^{\frac\alpha2},
 \]
  which is the desired H\"older continuity of $\wu$. 

  \medskip
Now, assume $u_0,u_1\in C^{1,1}(\overline \Omega)$ and $u_0=u_1$ on $\partial \Omega$. In this case, we do not need $\eta_{z_0}$ to construct lower barriers. We can use, instead of $V_{z_0}$ defined in \eqref{eq:lbarrier}, the subgeodesic 
      \[
      W(\wz)=\max\{V(\wz), u_0(z)+A\rho(z)\}
      \]
      where $V$ is defined in \eqref{eq:subg V},  $\rho$ is a strictly plurisubharmonic function defining $\Omega$, and $A>0$ is big enough such that $u_0+A\rho\leq u_1$ in $\Omega$. The function $W$ coincides with $\wu$ on $\partial D$ and is, as well as the upper barrier $\Psi$, Lipschitz continuous, which, using the same arguments as above, implies Lipschitz continuity of $\wu$ in $\overline{D}$. \\
More generally,  if $u_1=u_0+C$ on $\partial \Omega$ for a constant $C$, we can use the subgeodesic 
\[W(\wz)=\max\{V(\wz), u_0(z)+A\rho(z)+C\log|\zeta|\}.\]
\end{proof}

\begin{remark}
    In the proof of Theorem~\ref{thm:1/2 H}, the drop in regularity is caused by using 
the plurisubharmonic solution $\eta_{z_0}$ of 
$$(dd^c \eta)^n=0 \ {\rm in}\ \Omega,\quad \eta(z)=-|z-z_0|^\alpha  \ {\rm on}\ \partial\Omega.$$ 
While $\eta_{z_0}$ is $\alpha$-H\"older on $\partial\Omega$, it is only $\frac\alpha2$-H\"older in $\Omega$. Indeed, if $\Omega$ is the unit ball in $\C^2$ and $z_0=(1,0)$, then 
$$\eta_{z_0}(z_1,z_2)=-(2-z_1-\bar z_1)^{\frac\alpha2}.$$
\end{remark}
    
By Theorem \ref{thm:DRrepresentation} and by Arzelà-Ascoli, we get:

\begin{corollary}\label{cor:H env}
 Let $\Omega$ be strongly pseudoconvex. We have
 \begin{itemize}
     \item If $u,v\in\PSHc\Omega$ be $\alpha$-H\"older continuous in $\overline\Omega$, $0<\alpha\le 1$, then $P(u,v)$ is $\frac\alpha2$-H\"older continuous in $\overline\Omega$. 
\item If $u,v\in\PSH\Omega\cap C^{1,1}(\overline\Omega)$, $u-v=const$ on $\partial\Omega$, then $P(u,v)$ is Lipschitz in $\overline\Omega$.
  \end{itemize}
\end{corollary}

\subsection{$C^{1,1}$-regularity of geodesics in $\B_n$}
Here, we restrict ourselves to the case $\Omega=\B_n$, the unit ball of $\Cn$.

By following \cite{GZ12}, \cite[Chapter 5.3]{GZ}, which are, in turn, based on \cite{BT76}, we prove the following:

\begin{theorem}\label{thm:C11 geod}
   Let $u_0,u_1\in C^{1,1}(\overline\B_n)$. Then the function $\wu$ has second-order partial derivatives in the $z$-variables almost everywhere in $\B_n\times\cA$, and the derivatives are locally bounded. More precisely, for any $K\Subset\B_n$,
   \[
   \left|\frac{\partial^2 \wu(z,\zeta)}{\partial z_j\partial \bar z_k}
   \right|\le \frac{C}{{\rm dist}\, (K,\partial\B_n)^2}, \quad z\in K,\ \zeta\in \cA
   \]
   for a positive constant $C>0$ independent of $K$, $z$ and $\zeta$.
\end{theorem}

\begin{proof}
Given a holomorphic automorphism $\varphi$ of $\B_n$, the map $\hat\varphi: (z,\zeta)\mapsto (\varphi(z),\zeta)$ is a holomorphic automorphism of $D:=\B_n\times\cA$. So, the standard automorphism $\varphi_a$ with  $\varphi_a(a)=0$ and $\varphi_a^2={\rm Id}$ generates the corresponding automorphism $\hat\varphi_a$ of $D$. 

It is known that, when $|a|\le 1/2$,
\beq\label{eq:phia}
\varphi_a(z)=\frac{z-a+O(|a|^2)}{1-\langle z,a\rangle}
=z-a+\langle z,a\rangle z  +O(|a|^2)
= z-h+O(|a|^2),
\eeq
where $h=h(a):=a-\langle z,a\rangle z$ and $O(|a|^2)\le C_1|a|^2$ with $C_1$ independent of $z\in\overline\B_n$. 

The map $a\mapsto h$ is a diffeomorphism in a neighborhood of $0$ for $z\in\B_n$, and the inverse map $h\mapsto a$ has the norm bounded by $(1-|z|^2)^{-1}$ (cf. \cite[Chapter 5.3 page 146]{GZ}) Note also that $h(-a)=-h(a)$.

Consider the function
\[
\wv(\wz):=\frac12\left[\wu(\hat\varphi_a(\wz))+\wu(\hat\varphi_{-a}(\wz))\right]\in\PSHc{D}
\]
with boundary values 
\[
\psi(\wz):=\frac12\left[\Psi(\hat\varphi_a(\wz))+\Psi(\hat\varphi_{-a}(\wz))\right],\quad \wz\in\partial D,
\]
where $\Psi=\Psi_u$ is defined on $\overline D$ by (\ref{eq:Psi}).

Extending $u_j$ as a $C^{1,1}$-smooth function to a neighborhood $\B'$ of $\overline \B_n$, we get a $C^{1,1}$-extension of $\Psi$ to $D'=\B'\times\cA$ by the formula (\ref{eq:Psi}) such that
\[
|\Psi(\hat\varphi_a(\wz))-\Psi(\wz-\hat h)|\le C_2 |\hat\varphi_a(\wz)-(\wz-\hat h)|\leq  C_3|a|^2,
\]
where $\hat h=(h,0)$, with $C_2,C_3>0$ independent of $\wz$. Then using Taylor's formula and the fact that $|h|\le 2|a|$ (recall that $|z|\leq 1$) we obtain
\[
\Psi(\hat\varphi_{\pm a}(\wz))
\le \Psi(\wz\mp \hat h) + C_3|a|^2  \leq \Psi(\wz)+C_4|a|^2 
\]
with $C_4>0$ independent of $\wz$.
This gives 
\beq\label{eq:psi to Psi}
\wv(\wz)=\psi(\wz)\le  \Psi(\wz) + C_3|a|^2, \qquad \wz\in\partial\B_n\times\cA
\eeq

Also, on $\B_n\times\cA_j$, we have $\wv(\wz)=\frac12[u_j(\varphi_a(z))+u_j(\varphi_{-a}(z))]$ and, from the continuity of $u_j$ and once again Taylor's formula we get
\[
u_j(\varphi_{\pm a}(z))
\le u_j(z\mp h) + C_3|a|^2\leq u_j(z) + C_5|a|^2,
\]
with $C_5>0$ independent of $\wz$. Combining the above inequality (\ref{eq:psi to Psi}) we get
$$\wv(\wz)\leq \Psi(\wz) +C_6 |a|^2, \qquad \wz\in\partial(\B_n\times\cA),$$
for $C_6=\max\{C_4, C_5\}$.

We can then infer that the boundary values of the function $\wv-C_6|a|^2\in\PSHc{D}$ do not exceed those of $\wu$ and so, as the latter is maximal in $D$, we get $$\wv-C_6|a|^2\le\wu, \qquad \wz\in \overline{D}.$$ 
Equivalently,

\[
\wu(\hat\varphi_a(\wz))+\wu(\hat\varphi_{-a}(\wz)) -2\wu(\wz)\le C_6|a|^2,\quad \wz\in\overline{D}.
\]
From here, following word by word \cite[page 146]{GZ}, one deduces
\[
\wu(\wz+\hat h)+\wu(\wz-\hat h)-2\wu(\wz)\le \frac{C_7}{(1-|z|^2)^2} |h|^2
\]
with $C_7$ independent of $\wz$, which implies, by \cite[Lemma 2.11]{GZ12}, that $\wu$ has second-order partial derivatives in the $z$-variables almost everywhere in $D$, and the derivatives are locally bounded (see also \cite[page 147]{GZ}).
\end{proof}

\subsection{$C^{1,\bar 1}$-regularity of envelopes in $\B_n$}

As an consequence of Theorems  \ref{thm:DRrepresentation} and \ref{thm:C11 geod}, we get

\begin{corollary}\label{thm:smooth env}
   If $u,v\in C^{1,1}(\overline\B_n)$, then for any $K\Subset \B_n$ there exists $C_K>0$ such that $\Delta P(u,v)<C_K$ for all $z\in K$. In particular $P(u,v)\in C^{1,\bar 1}(\B_n)$.
\end{corollary}

\begin{proof}
    Let $u_t$ be the geodesic between $u_0=u$ and $u_1=v$. By Theorem \ref{thm:C11 geod}, for any $K\Subset \Omega$ there exists $C_K>0$ such that $0\le\Delta u_t<C_K$ for all $t\in[0,1]$.\\
    Also, Theorem \ref{thm:DRrepresentation} ensures that for any $z\in \B_n$, $P(u_0,u_1)(z)=\inf_{t\in(0,1)}u_t(z) $. It then follows from \cite[Proposition 4.4]{DR} that
    $\Delta P(u,v)\le C_K$, as we wanted. 
 \end{proof}
 
 A natural question one can ask is: {\sl Is it true that $P(u,v)\in C^{1,1}(\B_n)$?}

    \medskip

By following the lines of the proof of  Theorem~\ref{thm:C11 geod} and using Poletsky's theory of analytic disk functionals \cite{Pol93}, \cite{Pol93} (c.f. Remark~\ref{rem:analyt disks}), one can prove $C^{1,1}$-regularity of the envelopes $P(\psi)$ in $\B_n$ for arbitrary $\psi\in C^{1,1}(\overline{\B_n})$. Yet, the above question still remains open since the minimum of $C^{1,1}$ functions is not $C^{1,1}$.

 \begin{theorem}\label{thm:smooth env 2}
   If $\psi\in C^{1,1}(\overline\B_n)$, then  $P(\psi)\in C^{1, 1}(\B_n)$.
\end{theorem}

\begin{proof}
    We will use notations from the proof of Theorem~\ref{thm:C11 geod}, and we will again extend $\psi$ to a neighborhood of $\overline{\B_n}$. 

As mentioned in Remark~\ref{rem:analyt disks}, 
    \[
    P(\psi)(z)=\inf_{f\in A_z}\int_{\partial\D}f^*\psi\,d\sigma,\quad z\in\Omega,
    \]
where $A_z$ is the collection of all analytic disks $f:\overline\D\to\B_n$, $f(0)=z$, and $\sigma$ is the normalized Lebesgue measure on $\partial\D$.
If $f\in A_z$ and $a\in\B_n$, then the twisted analytic disk $f_a:=f^*\varphi_a\in A_{z_a}$ with $\varphi_a$ the holomorphic automorphism of $\B_n$ in \eqref{eq:phia} such that $z_a={\varphi_a(z)}$, so
\beq\label{eq:Poisson1}
    P(\psi)(z_a) =\inf_{f\in A_{z}}\int_{\partial\D}f_a^*\psi\,d\sigma, \quad z\in\Omega.
    \eeq
When $|a|<1/2$, we have, as in the proof of Theorem~\ref{thm:C11 geod}, the relation 
\[
|(f_a^*\psi)(\zeta)- \psi(f(\zeta)-h)|\le C_2 |f_a(\zeta)-(f(\zeta)-h)|\leq  C_3|a|^2,\quad \zeta\in\partial\D,
\]
 where $h=a-\langle z,a\rangle z$ and $O(|a|^2)\le C_1|a|^2$ with $C_1$ independent of $z\in\overline\B_n$, the map $a\mapsto h$ being a diffeomorphism in a neighborhood of $0$ for $z\in\B_n$, and the inverse map $h\mapsto a$ has the norm bounded by $(1-|z|^2)^{-1}$, see \cite[Chapter 5.3]{GZ}. 
 Replacing $a$ with $-a$, we get
 \[    
 (f_{\pm a}^*\psi)(\zeta) \le \psi(f(\zeta)\mp h) +  C_3|a|^2,\quad \zeta\in\partial\D
 \]
so that 
\begin{eqnarray*}
 (f_{a}^*\psi)(\zeta)+(f_{-a}^*\psi)(\zeta)
 & \le & 
 \psi(f(\zeta)- h) +  \psi(f(\zeta)+ h)+2 C_3|a|^2 \\
 &\le & (2f^*\psi)(\zeta) +C_4|a|^2,\quad \zeta\in\partial\D,
 \end{eqnarray*}
 the last inequality resulting from the $C^{1,1}$-regularity of $\psi$ in a neighborhood of $\overline{\B_n}$.

 Using \eqref{eq:Poisson1}, we derive
 \begin{eqnarray*}
     P(\psi)(z_a)+P(\psi)(z_{-a}) & \le & \inf_{f\in A_{z}}\int_{\partial\D}\left[(f_{a}^*\psi)+(f_{-a}^*\psi)\right]\,d\sigma \\
 & \le & \inf_{f\in A_{z}}\int_{\partial\D}
  (2f^*\psi)\,d\sigma +C_4|a|^2 \\
 & = & 2P(\psi)(z)+C_4|a|^2.
 \end{eqnarray*}
 From here, following word by word \cite[page 146]{GZ}, one deduces
\[
P(\psi)(z+h)+P(\psi)(z-h)-2P(\psi)(z)\le \frac{C_6}{(1-|z|^2)^2} |h|^2
\]
with $C_5$ independent of $z$, which implies $C^{1,1}$-regularity (see for example \cite[Lemma 2.11]{GZ12}) of $P(\psi)$ on any $\Omega'\Subset\B_n$.
\end{proof}
   

\end{document}